\documentclass[a4paper,11pt]{article}
\usepackage{amsmath, amsfonts, amscd, amssymb, amsthm, enumerate, xypic}

\DeclareMathOperator{\Mat}{\operatorname{M}}
\DeclareMathOperator{\GL}{\operatorname{GL}}

\DeclareMathOperator{\id}{\operatorname{id}}

\newcommand{\Ker}{\operatorname{Ker}}
\newcommand{\Vect}{\operatorname{Span}}
\newcommand{\im}{\operatorname{Im}}

\newcommand{\Sp}{\operatorname{Sp}}
\newcommand{\OSp}{\operatorname{OSp}}
\newcommand{\Span}{\operatorname{Span}}

\renewcommand{\setminus}{\smallsetminus}

\def\R{\mathbb{R}}
\def\C{\mathbb{C}}
\def\Q{\mathbb{Q}}
\def\N{\mathbb{N}}
\def\Z{\mathbb{Z}}

\def\U{\mathbb{U}}

\def\calP{\mathcal{P}}

\def\lcro{\mathopen{[\![}}
\def\rcro{\mathclose{]\!]}}

\theoremstyle{definition}
\newtheorem{Def}{Definition}
\newtheorem{Not}[Def]{Notation}

\theoremstyle{plain}
\newtheorem{theo}[Def]{Theorem}
\newtheorem{prop}[Def]{Proposition}

\newtheorem{lemme}[Def]{Lemma}

\theoremstyle{plain}

\theoremstyle{remark}

\title{On commuting matrices and exponentials}
\author{Cl\'ement de Seguins Pazzis\footnote{Lyc\'ee Priv\'e Sainte-Genevi\`eve, 2, rue
de l'\'Ecole des Postes, 78029 Versailles Cedex, FRANCE.}
\footnote{e-mail: dsp.prof@gmail.com}}

\begin{document}

\thispagestyle{plain}
\maketitle

\begin{abstract}
Let $A$ and $B$ be matrices of $\Mat_n(\C)$. We show that if $\exp(A)^k \exp(B)^l=\exp(kA+lB)$ for all integers $k$ and $l$,
then $AB=BA$.
We also show that if $\exp(A)^k \exp(B)=\exp(B)\exp(A)^k=\exp(kA+B)$ for every positive integer $k$, then the pair $(A,B)$
has property L of Motzkin and Taussky. \\
As a consequence, if $G$ is a subgroup of $(\Mat_n(\C),+)$ and $M \mapsto \exp(M)$ is a homomorphism from $G$ to $(\GL_n(\C),\times)$,
then $G$ consists of commuting matrices. If $S$ is a subsemigroup of $(\Mat_n(\C),+)$ and $M \mapsto \exp(M)$ is a homomorphism from $S$ to $(\GL_n(\C),\times)$, then the linear subspace $\Span(S)$ of $\Mat_n(\C)$ has property L of Motzkin and Taussky.
\end{abstract}



\section{Introduction}

\subsection{Notation and definition}

\begin{enumerate}[i)]
\item We denote by $\N$ the set of non-negative integers.
\item If $M \in \Mat_n(\C)$, we denote by $e^M$ or $\exp(M)$ its exponential, by $\Sp(M)$ its set of eigenvalues.
\item The $n \times n$ complex matrices $A$, $B$ are said to be simultaneously triangularizable if there exists
an invertible matrix $P$ such that $P^{-1}AP$ and $P^{-1}BP$ are upper triangular.
\item A pair $(A,B)$ of complex $n \times n$ matrices is said to have property L if for a special ordering
$(\lambda_i)_{1 \leq i \leq n}$, $(\mu_i)_{1 \leq i \leq n}$ of the eigenvalues of $A$, $B$, the eigenvalues of $xA+yB$ are
$(x\lambda_i+y\mu_i)_{1 \leq i \leq n}$ for all values of the complex numbers $x$, $y$.
\end{enumerate}

\subsection{The problem}

It is well known that the exponential is not a group homomorphism from $(\Mat_n(\C),+)$
to $(\GL_n(\C),\times)$ if $n \geq 2$. Nevertheless,
when $A$ and $B$ are commuting matrices of $\Mat_n(\C)$, one has
\begin{equation}\label{basic}
e^{A+B}=e^Ae^B=e^Be^A.
\end{equation}
However \eqref{basic} is not a sufficient condition for the commutativity of $A$ with $B$, nor even for
$A$ and $B$ to be simultaneously triangularizable.
Still, if
\begin{equation}\label{local}
\forall t \in \R, \; e^{tA}e^{tB}=e^{tB}e^{tA},
\end{equation}
or
\begin{equation}\label{local2}
\forall t \in \R, \; e^{t(A+B)}=e^{tA}e^{tB},
\end{equation}
then a power series expansion at $t=0$ shows that $AB=BA$.
In the 1950s, pairs of matrices $(A,B)$ of small size such that $e^{A+B}=e^Ae^B$
have been under extensive scrutiny \cite{Frechet1,Frechet2,Huff,Kakar,MoNo1,MoNo2}.
More recently, Wermuth \cite{WerI,WerII} and Schmoeger \cite{SchmoegerI,SchmoegerII} studied the problem of adding extra conditions on the matrices $A$ and $B$ for
the commutativity of $e^A$ with $e^B$ to imply the commutativity of $A$ with $B$.
A few years ago, Bourgeois (see \cite{Bourgeois}) investigated, for small $n$,
the pairs $(A,B)\in \Mat_n(\C)^2$ that satisfy
\begin{equation}
\label{Bourgeoiscond}
\forall k \in \N, \quad e^{kA+B}=e^{kA}e^B=e^B e^{kA}.
\end{equation}
The main interest in this condition lies in the fact that, contrary to conditions \eqref{local} and \eqref{local2}, it is not possible to use
it to obtain information on $A$ and $B$ based only on the local behavior of the exponential around $0$.
Bourgeois showed that Condition \eqref{Bourgeoiscond} implies that $A$ and $B$ are simultaneously triangularizable if $n=2$, and produced a proof
that this also holds when $n=3$. This last result is however false, as the following counterexample (communicated to us by Jean-Louis Tu) shows:
consider the matrices
$$A_1:=2i\pi\begin{bmatrix}
1 & 0 & 0 \\
0 & 2 & 0 \\
0 & 0 & 0
\end{bmatrix} \quad \text{and} \quad
B_1:=2i\pi\begin{bmatrix}
2 & 1 & 1 \\
1 & 3 & -2 \\
1 & 1 & 0
\end{bmatrix}.$$
Notice that $A_1$ and $B_1$ are not simultaneously triangularizable since they
share no eigenvector (indeed, the eigenspaces of $A_1$ are the lines spanned
by the three vectors of the canonical basis, and none of them is stabilized by $B_1$).
However, for every $t \in \C$, a straightforward computation shows that the characteristic polynomial of
$tA_1+B_1$ is
$$X\bigl(X-2i\pi(t+2)\bigr)\bigl(X-2i\pi(2t+3)\bigr).$$
Then for every $t \in \N$, the matrix $tA_1+B_1$ has three distinct eigenvalues in $2i\pi \Z$,
hence is diagonalizable with $e^{tA_1+B_1}=I_3$. In particular $e^{B_1}=I_3$, and on the other hand
$e^{A_1}=I_3$. This shows that Condition \eqref{Bourgeoiscond} holds.

\vskip 2mm
It then appears that one should strengthen Bourgeois' condition as follows in order
to obtain at least the simultaneous triangularizability of $A$ and $B$:
\begin{equation}
\label{dSPcond}
\forall (k,l) \in \Z^2, \quad e^{kA+lB}=e^{kA}e^{lB}.
\end{equation}
Notice immediately that this condition implies that $e^A$ and $e^B$ commute.
Indeed, if Condition \eqref{dSPcond} holds, then
$$e^Be^A=\bigl(e^{-A}e^{-B}\bigr)^{-1}=\bigl(e^{-A-B}\bigr)^{-1}=e^{A+B}=e^A e^B.$$
Therefore Condition \eqref{dSPcond} is equivalent to
\begin{equation}
\label{dSPcond2}
\forall (k,l) \in \Z^2, \quad e^{kA+lB}=e^{kA}e^{lB}=e^{lB}e^{kA}.
\end{equation}

Here is our main result.
\begin{theo}\label{2matrices}
Let $(A,B)\in \Mat_n(\C)^2$ be such that, for all $(k,l)\in \Z^2$, $e^{kA+lB}=e^{kA}e^{lB}$.
Then $AB=BA$.
\end{theo}

The following corollary is straightforward.

\begin{theo}\label{group}
Let $G$ be a subgroup of $(\Mat_n(\C),+)$ and assume that $M \mapsto \exp(M)$
is a homomorphism from $(G,+)$ to $(\GL_n(\C),\times)$.
Then, for all $(A,B)\in G^2$, $AB=BA$.
\end{theo}

The key of the proof of Theorem \ref{2matrices} is

\begin{prop}\label{spectredansZ}
Let $(A,B)\in \Mat_n(\C)^2$. Assume that, for every $(k,l)\in \Z^2$, the matrix
$kA+lB$ is diagonalizable and $\Sp(kA+lB) \subset \Z$. Then $AB=BA$.
\end{prop}

For subsemigroups of $(\Mat_n(\C),+)$, Theorem \ref{group} surely fails.
A very simple counterexample is indeed given by the semigroup generated by
$$A:=\begin{bmatrix}
0 & 0 \\
0 & 2i\pi
\end{bmatrix} \quad \text{and} \quad B:=\begin{bmatrix}
0 & 1 \\
0 & 2i\pi
\end{bmatrix}.$$
One may however wonder whether a subsemigroup $S$ on which the exponential is a homomorphism must be simultaneously
triangularizable. Obviously the additive semigroup generated by the matrices $A_1$ and $B_1$ above is a counterexample.
Nevertheless, we will prove a weaker result, which rectifies and generalizes Bourgeois' results \cite{Bourgeois}.

\begin{prop}\label{Bourgeoistheo}
Let $(A,B) \in \Mat_n(\C)^2$ be such that $\forall k \in \N, \; e^{kA+B}=e^{kA}e^B=e^Be^{kA}$.
Then $(A,B)$ has property~L.
\end{prop}

Note that the converse is obviously false.

\vskip 3mm
The proofs of Theorem \ref{2matrices} and of Proposition \ref{Bourgeoistheo}
have largely similar parts, so they will be tackled simultaneously.
There are three main steps.
\begin{itemize}
\item We will prove Proposition \ref{Bourgeoistheo} in the special case where
$\Sp(A) \subset 2i\pi \Z$ and $\Sp(B) \subset 2i\pi \Z$. This will involve a study of the
matrix pencil $z \mapsto A+zB$. We will then easily derive Proposition \ref{spectredansZ}
using a refinement of the Motzkin-Taussky theorem.
\item We will handle the more general case $\Sp(A) \subset 2i\pi\Z$ and $\Sp(B) \subset 2i\pi \Z$ in Theorem \ref{2matrices}
by using the Jordan-Chevalley decompositions of $A$ and $B$ together with Proposition \ref{spectredansZ}.
\item In the general case, we will use an induction to reduce the situation to the previous one, both for Theorem \ref{2matrices}
and Proposition \ref{Bourgeoistheo}.
\end{itemize}
In the last section, we will prove a sort a generalized version of Proposition \ref{Bourgeoistheo} for additive semigroups of matrices
(see Theorem \ref{semigroup}).

\section{Additive groups and semigroups of matrices with an integral spectrum}

\subsection{Notation}

\begin{enumerate}[i)]
\item We denote by $\Sigma_n$ the group of permutations of $\{1,\dots,n\}$, make
it act on $\C^n$ by $\sigma.(z_1,\dots,z_n):=(z_{\sigma(1)},\dots,z_{\sigma(n)})$,
and consider the quotient set $\C^n/\Sigma_n$. The class of a list $(z_1,\dots,z_n)\in \C^n$
in $\C^n/\Sigma_n$ will be denoted by $[z_1,\dots,z_n]$.
\item For $M \in \Mat_n(\C)$, we denote by $\chi_M(X) \in \C[X]$ its characteristic polynomial, and we set
$$\OSp(M):=[z_1,\dots,z_n], \quad \text{where $\chi_M(X)=\prod_{k=1}^n (X-z_k)$.}$$
\item Given an integer $N \geq 1$, we set $\U_N(z):=\bigl\{\zeta \in \C : \; \zeta^N=z\bigr\}$.
\end{enumerate}

\subsection{Definition}

\begin{Def}[A reformulation of Motzkin-Taussky Property L \cite{MoTau}] ${}$ \\
A pair $(A,B) \in \Mat_n(\C)^2$ has property L when there are $n$ linear forms $f_1,\dots,f_n$ on $\C^2$ such that
$$\forall (x,y)\in \C^2, \; \OSp(xA+yB)=\bigl[f_k(x,y)\bigr]_{1 \leq k \leq n}.$$
Using the fact that the eigenvalues are continuous functions of the coefficients, it is obvious that
a pair $(A,B) \in \Mat_n(\C)^2$ has property L if and only if
there are affine maps $f_1,\dots,f_n$ from $\C$ to $\C$ such that
$$\forall z \in \C, \; \OSp(A+zB)=\bigl[f_k(z)\bigr]_{1\leq k \leq n}.$$
\end{Def}

\subsection{Property L for pairs of matrices with an integral spectrum}

We denote by $\mathcal{K}(\C)$ the quotient field of the integral domain $H(\C)$ of entire functions (i.e.\ analytic functions from $\C$ to $\C$).
Considering $\id_\C$ as an element of $\mathcal{K}(\C)$, we may view $A+\id_\C B$ as a matrix of $\Mat_n\bigl(\mathcal{K}(\C)\bigr)$.
We define the \textbf{generic number} $p$ of eigenvalues of the pencil $z \mapsto A+zB$
as the number of the distinct eigenvalues of $A+\id_\C B$ in an algebraic closure of $\mathcal{K}(\C)$.
A complex number $z$ is called \textbf{regular} when $A+zB$ has exactly $p$ distinct eigenvalues, and
\textbf{exceptional} otherwise.
In a neighborhood of $0$, the spectrum of $A+zB$ may be classically described with Puiseux series as follows (see \cite[chapter 7]{Fischer}):
there exists a radius $r>0$, an integer $q \in \{1,\dots,n\}$, positive integers
$d_1,\dots,d_q$ such that $n=d_1+\dots+d_q$, and analytic functions $f_1,\dots,f_q$ defined on a neighborhood of $0$
such that
$$\forall z \in \C \setminus \{0\}, \quad |z|<r \,\Rightarrow\, \chi_{A+zB}(X)=\underset{k=1}{\overset{q}{\prod}}\,\underset{\zeta \in \U_{d_k}(z)}{\prod}
\bigl(X-f_k(\zeta)\bigr).$$
\vskip 2mm
\noindent We may now prove the following result.
\begin{prop}\label{BourgeoisimpliqueLrestr}
Let $(A,B)\in \Mat_n(\C)^2$. Assume that $\Sp(kA+B) \subset \Z$ for every $k \in \N$.
Then $(A,B)$ has property L.
\end{prop}

\begin{proof}
With the above notation, we prove that $f_1,\dots,f_q$ are polynomial functions. For instance, consider $f_1$
and its power series expansion
$$f_1(z)=\underset{j=0}{\overset{+\infty}{\sum}} a_j z^j.$$
Set $N:=d_1$ for convenience.
Let $k_0$ be a positive integer such that $\frac{1}{k_0}<r$.
For every integer $k \geq k_0$, $kf_1\bigl(k^{-\frac{1}{N}}\bigr)$ is an eigenvalue of $kA+B$: hence it is an integer.
Therefore one has: for every integer $k \geq k_0$,
$$(k+1)f_1\Bigl((k+1)^{-\frac{1}{N}}\Bigr)-kf_1\Bigl(k^{-\frac{1}{N}}\Bigr) \in \Z.$$
For every integer $k \geq k_0$, the following equality holds:
$$(k+1)f_1\Bigl((k+1)^{-\frac{1}{N}}\Bigr)-kf_1\Bigl(k^{-\frac{1}{N}}\Bigr)=a_0+
\underset{j \in \N \setminus \{0,N\}}{\sum}a_j\bigl((k+1)^{1-\frac{j}{N}}-k^{1-\frac{j}{N}}\bigr).$$
Assume that $a_j \neq 0$ for some $j \geq 1$ with $j \neq N$, and define $s$ as the smallest such $j$.
On the one hand, one has for every integer $j \in \N$,
$$(k+1)^{1-\frac{j}{N}}-k^{1-\frac{j}{N}}=k^{1-\frac{j}{N}}\Bigl(\Bigl(1+\frac{1}{k}\Bigr)^{1-\frac{j}{N}}-1\Bigr)
\underset{k \rightarrow +\infty}{\sim} k^{1-\frac{j}{N}}\,\frac{1-\frac{j}{N}}{k}=\Bigl(1-\frac{j}{N}\Bigr)\, k^{-\frac{j}{N}}.$$
On the other hand, when $k \rightarrow +\infty$, one has
\begin{eqnarray*}
\underset{j=s+N+1}{\overset{+\infty}{\sum}}a_j k^{1-\frac{j}{N}} & = & o\Bigl(k^{-\frac{s}{N}}\Bigr) \quad \text{and} \\
\underset{j=s+N+1}{\overset{+\infty}{\sum}}a_j (k+1)^{1-\frac{j}{N}} & = & o\Bigl(k^{-\frac{s}{N}}\Bigr).
\end{eqnarray*}
It follows that
$$\underset{j \in \N \setminus \{0,N\}}{\sum}a_j\bigl((k+1)^{1-\frac{j}{N}}-k^{1-\frac{j}{N}}\bigr)
\underset{k \rightarrow +\infty}{\sim} a_s \Bigl(1-\frac{s}{N}\Bigr)\, k^{-\frac{s}{N}}.$$
The sequence $\Bigl((k+1)f_1\bigl((k+1)^{-\frac{1}{N}}\bigr)-kf_1\bigl(k^{-\frac{1}{N}}\bigr)-a_0\Bigr)_{k \geq k_0}$ is discrete,
converges to $0$ and is not ultimately zero.
This is a contradiction. Therefore $\forall j \in \N \setminus \{0,N\}, \; a_j=0$.
In the same way, one shows that, for every $k \in \{1,\dots,q\}$, there exists a $b_k \in \C$ such that
$f_k(z)=f_k(0)+b_k z^{d_k}$ in a neighborhood of $0$.
It follows that, in a neighborhood of $0$,
$$\chi_{A+zB}(X)=\underset{k=1}{\overset{q}{\prod}}(X-f_k(0)-b_kz)^{d_k}.$$
Therefore we found affine maps $g_1,\dots,g_n$ from $\C$ to $\C$ such that, in a neighborhood of $0$,
$$\chi_{A+zB}(X)=\underset{k=1}{\overset{n}{\prod}}\bigl(X-g_k(z)\bigr).$$
The coefficients of these polynomials are polynomial functions of $z$
that coincide on a neighborhood of $0$; therefore
$$\forall z \in \C, \; \chi_{A+zB}(X)=\underset{k=1}{\overset{n}{\prod}} \bigl(X-g_k(z)\bigr).$$
The pair $(A,B)$ has property L, and Proposition \ref{BourgeoisimpliqueLrestr} is proven.
\end{proof}

\subsection{Commutativity for subgroups of diagonalizable matrices with an integral spectrum}

Given a matrix $M \in \Mat_n(\C)$ and an eigenvalue $\lambda$ of it,
recall that the \textbf{eigenprojection} of $M$ associated to $\lambda$ is the projection onto $\Ker(M-\lambda\,I_n)^n$ alongside
$\im(M-\lambda\,I_n)^n=\underset{\mu \in \Sp(M),\,\mu \neq \lambda}{\sum} \Ker(M-\mu\,I_n)^n$.

\vskip 2mm
Here, we derive Proposition \ref{spectredansZ} from Proposition \ref{BourgeoisimpliqueLrestr}.
We start by explaining how Kato's proof \cite[p.85 Theorem 2.6]{Kato} of the Motzkin-Taussky theorem \cite{MoTauII}
leads to the following refinement.

\begin{theo}[Refined Motzkin-Taussky theorem]\label{refinedMoTau}
Let $(A,B)\in \Mat_n(\C)^2$ be a pair of matrices which satisfies property L.
Assume that $B$ is diagonalizable and that $A+z_0B$ is diagonalizable for every exceptional point $z_0$ of the matrix pencil $z \mapsto A+zB$.
Then $AB=BA$.
\end{theo}

\begin{proof}
We refer to the line of reasoning of \cite[p.85 Theorem 2.6]{Kato} and explain how it may be adapted to prove Theorem \ref{refinedMoTau}. Denote by $p$ the generic number of eigenvalues of $z \mapsto A+zB$,
and by $f_1,\dots,f_p$ the $p$ distinct affine maps such that $\forall z \in \C, \; \Sp(A+zB)=\{f_1(z),\dots,f_p(z)\}$.
Denote by $\Omega$ the (open) set of regular points of $z \mapsto A+zB$, i.e.\
$$\Omega=\C \setminus \bigl\{z \in \C : \; \exists (i,j)\in \{1,\dots,p\}^2 : \; i \neq j\; \text{and} \; f_i(z)=f_j(z)\bigr\}.$$
For $z \in \Omega$ and $i \in \{1,\dots,p\}$, denote by $\Pi_i(z)$ the eigenprojection of $A+zB$ associated to the eigenvalue
$f_i(z)$. Then  $z \mapsto \Pi_i(z)$ is holomorphic on $\Omega$ for any $i \in \{1,\dots,p\}$ (see \cite[II.1.4]{Kato}).
Let $z_0 \in \C \setminus \Omega$. Then $A+z_0B$ is diagonalizable and hence
\cite[p.82, Theorem 2.3]{Kato} shows that $z_0$ is a regular point for each map $z \mapsto \Pi_i(z)$.
We deduce that the functions $(\Pi_i)_{i \leq p}$ are restrictions of entire functions.
Since $B$ is diagonalizable, these functions are bounded at infinity (see the last paragraph of \cite[p.85]{Kato})
and Liouville's theorem yields that they are constant.
By a classical continuity argument (see \cite[II.1.4, formula (1.16)]{Kato}), we deduce that each eigenprojection
of $B$ is sum of some projections chosen among the $\bigl(\Pi_i(0)\bigr)_{i \leq p}$.
As $B$ is diagonalizable, it is a linear combination of the $\bigl(\Pi_i(0)\bigr)_{i \leq p}$, which all commute with $A+zB$
for any regular $z$. Therefore $AB=BA$.
\end{proof}

We now turn to the proof of Proposition \ref{spectredansZ}.

\begin{proof}[Proof of Proposition \ref{spectredansZ}]
Let $(A,B)\in \Mat_n(\C)^2$. Assume that, for every $(k,l)\in \Z^2$, the matrix
$kA+lB$ is diagonalizable and $\Sp(kA+lB) \subset \Z$.
Proposition \ref{BourgeoisimpliqueLrestr} then shows that $(A,B)$ has property L.
For $k \in \lcro 1,n\rcro$, choose $f_k : (y,z)  \mapsto \alpha_k y+\beta_k z$ such that
$$\forall (y,z) \in \C^2, \; \OSp(yA+zB)=\bigl[f_k(y,z)\bigr]_{1 \leq k \leq n}.$$
Since $\Sp(A)=\{\alpha_1,\dots,\alpha_n\}$ and $\Sp(B)=\{\beta_1,\dots,\beta_n\}$,
the families $(\alpha_k)_{k \leq n}$ and $(\beta_k)_{k \leq n}$ are made of integers. It follows that
the exceptional points of the matrix pencil $z \mapsto A+zB$ are rational numbers.
As the matrix $A+\frac{l}{k}\,B=\frac{1}{k}\,(k\,A+l\,B)$ is diagonalizable for every $(k,l)\in (\Z \setminus \{0\}) \times \Z$,
the refined Motzkin-Taussky theorem implies that $AB=BA$.
\end{proof}

We now deduce the following special case of Theorem \ref{2matrices}.

\begin{lemme}\label{lemmecasnoyau}
Let $(A,B) \in \Mat_n(\C)^2$ be such that for all $(k,l)\in \Z^2$,
$e^{kA+lB}=I_n$. Then $AB=BA$.
\end{lemme}

\begin{proof}
Recall that the solutions of the equation $e^M=I_n$ are the diagonalizable matrices $M$
such that $\Sp(M) \subset 2i\pi\Z$ (see \cite[Theorem 1.27]{Higham}).
In particular, for every $(k,l)\in \Z^2$, the matrix $kA+lB$ is diagonalizable and
$\Sp(kA+lB) \subset 2i\pi \Z$. Setting $A':=\frac{1}{2i\pi}\,A$ and $B':=\frac{1}{2i\pi}\,B$,
we deduce that $(A',B')$ satisfies the assumptions of Proposition \ref{spectredansZ}.
It follows that $A'B'=B'A'$, and hence $AB=BA$.
\end{proof}

\section{The case $\Sp(A) \subset 2i\pi\Z$ and $\Sp(B) \subset 2i\pi \Z$ in Theorem~\ref{2matrices}}\label{spectrelimite}

\begin{prop}\label{2eetape}
Let $(A,B)\in \Mat_n(\C)^2$ be such that $\forall (k,l)\in \Z^2, \; e^{kA+lB}=e^{kA}e^{lB}$,
$\Sp(A) \subset 2i\pi\Z$ and $\Sp(B) \subset 2i\pi \Z$. Then
$AB=BA$.
\end{prop}

\begin{proof}
We consider the Jordan-Chevalley decompositions $A=D+N$ and $B=D'+N'$, where $D$ and $D'$ are diagonalizable,
$N$ and $N'$ are nilpotent and $DN=ND$ and $D'N'=N'D'$.
Clearly, for every integer $k$, $kA=kD+kN$ (resp.\ $kB=kD'+kN'$) is the Jordan-Chevalley decomposition of $kA$ (resp.\ of $kB$),
and $\Sp(kD)=\Sp(kA)=k\,\Sp(A) \subset 2i\pi \Z$ (resp.\
$\Sp(kD')=\Sp(kB)=k\,\Sp(B) \subset 2i\pi \Z$). This shows that
$$e^{kA}=e^{kN} \quad \text{and} \quad e^{kB}=e^{kN'}.$$
Condition \eqref{dSPcond2} may be written as
$$\forall (k,l)\in \Z^2, \quad e^{kA+lB}=e^{kN}e^{lN'}=e^{lN'}e^{kN}.$$
Note in particular that $e^N$ and $e^{N'}$ commute. Since
$N$ is nilpotent, we have
$$N=\underset{k=1}{\overset{n-1}{\sum}}\frac{(-1)^{k+1}}{k}\bigl(e^N-I_n\bigr)^{k}.$$
That shows that
$N$ is a polynomial in $e^N$.
Similarly $N'$ is a polynomial in $e^{N'}$. Therefore,
$$NN'=N'N.$$
The above condition yields
$$\forall (k,l)\in \Z^2, \; e^{kA+lB}=e^{kN+lN'}.$$
For any $(k,l)\in \Z^2$, $kN+lN'$ is nilpotent since $N$ and $N'$ are commuting nilpotent matrices.
Hence $kN+lN'$ is a polynomial in $e^{kN+lN'}$. Since $kA+lB$ commutes with $e^{kA+lB}$, it commutes with $kN+lN'$.
Therefore
$$e^{kD+l D'}=e^{kA+lB}e^{-kN-lN'}=I_n.$$
In particular, this yields that $kD+lD'$ is diagonalizable with $\Sp(kD+lD') \subset 2i\pi \Z$,
and the Jordan-Chevalley decomposition of $kA+lB$ is
$kA+lB=(kD+lD')+(kN+lN')$ as $kN+lN'$ commutes with $kA+lB$.

\vskip 2mm
By Lemma \ref{lemmecasnoyau}, the matrices $D$ and $D'$ commute.
In particular $(D,D')$ has property L, which yields affine maps $f_1,\dots,f_n$ from $\C$ to $\C$
such that
$$\forall z \in \C, \; \OSp(D+zD')=\bigl[f_k(z)\bigr]_{1 \leq k \leq n}.$$
The set
$$E:=\bigl\{k \in \Z : \; \exists (i,j)\in \{1,\dots,n\}^2 : \; f_i \neq f_j \; \text{and}\; f_i(k)=f_j(k)\bigr\}$$
is clearly finite.
We may choose two distinct elements $a$ and $b$ in $\Z \setminus E$.
The following equivalence holds:
$$\forall (i,j)\in \{1,\dots,n\}^2, \;
f_i(a)=f_j(a) \Leftrightarrow f_i=f_j \Leftrightarrow f_i(b)=f_j(b).$$
Since $D$ and $D'$ are simultaneously diagonalizable, it easily follows that
$D+aD'$ is a polynomial in $D+bD'$ and conversely $D+bD'$ is a polynomial in $D+aD'$.
Hence $N+aN'$ and $N+bN'$ both commute with $D+aD'$ and $D+bD'$.
Since $N+aN'$ and $N+bN'$ both commute with one another,
we deduce that $A+aB=(D+aD')+(N+aN')$ commutes with $A+bB=(D+bD')+(N+bN')$.
Since $a \neq b$, we conclude that $AB=BA$.
\end{proof}

\section{Proofs of Theorem \ref{2matrices} and Proposition \ref{Bourgeoistheo}}

\begin{Def}\label{defsection4}
Let $(A,B)\in \Mat_n(\C)^2$.
\begin{enumerate}[i)]
\item $(A,B)$ is said to be decomposable if there exists a non-trivial decomposition
$\C^n=F \oplus G$ in which $F$ and $G$ are invariant linear subspaces for both $A$ and $B$.
\item In the sequel, we consider, for $k \in \N \setminus \{0\}$, the function
$$\gamma_k : (\lambda,\mu) \in \Sp(e^A) \times \Sp(e^B) \mapsto \lambda^k \mu \in \C.$$
\item For $\lambda \in \C$, we denote by $C_\lambda(M)$ the
characteristic subspace of $M$ with respect to $\lambda$, i.e.\ $C_\lambda(M)=\Ker(M-\lambda I_n)^n$.
\end{enumerate}
\end{Def}

\begin{lemme}\label{ultimlemma1}
Assume that $A$ satisfies Condition
\begin{equation}\label{condonA}
\forall (\lambda,\mu)\in \Sp(A)^2, \; \lambda-\mu \in 2i\pi\Q \Rightarrow
\lambda-\mu \in 2i\pi \Z.
\end{equation}
Then there exists $k \in \N \setminus \{0\}$ such that $\gamma_k$ is one-to-one.
\end{lemme}

\begin{proof}
Assume that for every $k \in \N \setminus \{0\}$, there are distinct pairs $(\lambda,\mu)$ and $(\lambda',\mu')$ in
$\Sp(e^A) \times \Sp(e^B)$ such that $\lambda^k \mu=(\lambda')^k\mu'$.
Since $\Sp(e^A) \times \Sp(e^B)$ is finite and $\N \setminus \{0\}$ is infinite, we may then find
distinct pairs $(\lambda,\mu)$ and $(\lambda',\mu')$ in
$\Sp(e^A) \times \Sp(e^B)$ and distinct non-zero integers $a$ and $b$ such that
$$\lambda^a \mu=(\lambda')^a\mu' \quad \text{and} \quad
\lambda^b \mu=(\lambda')^b \mu'.$$
All those eigenvalues are non-zero and $\bigl(\frac{\lambda}{\lambda'}\bigr)^{a-b}=1$ with $a \neq b$.
It follows that $\frac{\lambda}{\lambda'}$ is a root of unity.
However $\lambda=e^{\alpha}$ and $\lambda'=e^{\beta}$ for some $(\alpha,\beta)\in \Sp(A)^2$, which shows that $(a-b)(\alpha-\beta) \in 2i\pi \Z$. Condition \eqref{condonA} yields $\alpha-\beta \in 2i\pi \Z$; hence $\lambda=\lambda'$.
It follows that $\mu=\mu'$, in contradiction with $(\lambda,\mu) \neq (\lambda',\mu')$.
\end{proof}

\begin{lemme}\label{ultimlemma2}
Assume that $\gamma_1$ is one-to-one and that $(A,B)$ satisfies Equality
\eqref{dSPcond} (resp.\ Equality \eqref{Bourgeoiscond}). Then the characteristic subspaces of $e^A$ and $e^B$ are stabilized
by $A$ and $B$.
\end{lemme}

\begin{proof}
Notice that $A+B$ commutes with $e^{A+B}$, hence commutes with $e^Ae^B$.
It thus stabilizes the characteristic subspaces of $e^Ae^B$.
Let us show that
\begin{equation}\label{decomp}
\forall \mu\in \Sp(e^B), \; C_\mu(e^B)=\bigoplus_{\lambda \in \Sp(e^A)} C_{\lambda \mu}(e^Ae^B).
\end{equation}
$\bullet$ Since $e^B$ and $e^A$ commute, $e^A$ stabilizes the characteristic subspaces
of $e^B$. Considering the characteristic subspaces of the endomorphism of
$C_\mu(e^B)$ induced by $e^A$, we find
$$\forall \mu \in \Sp(e^B), \;
C_\mu(e^B)=\bigoplus_{\lambda \in \Sp(e^A)} \bigl[C_{\lambda}(e^A) \cap C_{\mu}(e^B)\bigr].$$
$\bullet$ Let $(\lambda,\mu)\in \Sp(e^A)\times \Sp(e^B)$.
Since $e^A$ and $e^B$ commute, they both stabilize $C_{\lambda}(e^A) \cap C_{\mu}(e^B)$
and induce simultaneously triangularizable endomorphisms of $C_{\lambda}(e^A) \cap C_{\mu}(e^B)$
each with a sole eigenvalue, respectively $\lambda$ and $\mu$: it follows that
$$C_{\lambda}(e^A) \cap C_{\mu}(e^B) \subset C_{\lambda \mu}(e^Ae^B).$$
$\bullet$ Finally, the application $(\lambda,\mu) \mapsto \lambda\mu$ is one-to-one on $\Sp(e^A)\times \Sp(e^B)$. Therefore
$$C_{\lambda \mu}(e^Ae^B) \cap
C_{\lambda' \mu'}(e^A e^B) =\{0\}$$
for all distinct pairs $(\lambda,\mu)$ and $(\lambda',\mu')$ in $\Sp(e^A) \times \Sp(e^B)$. \\
One has
$$\C^n=\bigoplus_{\mu \in \Sp(e^B)} C_\mu(e^B)=\bigoplus_{\mu \in \Sp(e^B)}\bigoplus_{\lambda \in \Sp(e^A)}\bigl[C_{\lambda}(e^A)
\cap C_\mu(e^B)\bigr]$$
and $\C^n$ is the sum of all the characteristic subspaces of $e^A e^B$. We deduce that
$$\forall (\lambda,\mu)\in \Sp(e^A)\times \Sp(e^B), \;
C_{\lambda \mu}(e^Ae^B)=C_{\lambda}(e^A) \cap C_\mu(e^B).$$
This gives Equality \eqref{decomp}.

\vskip 2mm
We deduce that $A+B$ stabilizes every characteristic subspace of $e^B$.
However this is also true of $B$ since it commutes with $e^B$.
Hence both $A$ and $B$ stabilize the characteristic subspaces of $e^B$.
Symmetrically, every characteristic subspace of $e^A$ is stabilized by both $A$ and $B$.
\end{proof}

\begin{proof}[Proof of Theorem \ref{2matrices} and Proposition \ref{Bourgeoistheo}]
We use an induction on $n$.
Both Theorem \ref{2matrices} and Proposition \ref{Bourgeoistheo}
obviously hold for $n=1$, so we fix $n \geq 2$ and assume that they hold for any
pair $(A,B)\in \Mat_k(\C)^2$ with $k \in \{1,\dots,n-1\}$.
Let $(A,B) \in \Mat_n(\C)^2$ satisfying Equality \eqref{dSPcond} (resp.\ Equality \eqref{Bourgeoiscond}).
Assume first that $(A,B)$ is decomposable. Then there exists $p \in \{1,\dots,n-1\}$, a non-singular matrix $P \in \GL_n(\C)$
and square matrices $A_1,B_1,A_2,B_2$ respectively in $\Mat_p(\C)$, in $\Mat_p(\C)$, in $\Mat_{n-p}(\C)$
and in $\Mat_{n-p}(\C)$ such that
$$A=P\begin{bmatrix}
A_1 & 0 \\
0 & A_2
\end{bmatrix}P^{-1} \quad \text{and} \quad
B=P\begin{bmatrix}
B_1 & 0 \\
0 & B_2
\end{bmatrix}P^{-1}.$$
Since the pair $(A,B)$ satisfies Equality \eqref{dSPcond} (resp.\ Equality \eqref{Bourgeoiscond}), it easily follows that this is also
the case of $(A_1,B_1)$ and $(A_2,B_2)$; hence the induction hypothesis yields that
$(A_1,B_1)$ and $(A_2,B_2)$ are commuting pairs (resp.\ have property L). Therefore $(A,B)$ is also
a commuting pair (resp.\ has property L).

\vskip 3mm
From that point on, we assume that $(A,B)$ is indecomposable.
We may also assume that $A$ satisfies Condition \eqref{condonA}.
Indeed, consider in general the finite set
$$\mathcal{E}:=\Q \cap \frac{1}{2i\pi}\,\bigl\{\lambda-\mu \mid (\lambda,\mu)\in \Sp(A)^2\bigr\}.$$
Since its elements are rational numbers, we may find some integer $p>0$
such that $p\mathcal{E}\subset \Z$. Replacing $A$ with $pA$, we notice that
$(pA,B)$ still satisfies Equality \eqref{dSPcond} (resp.\ Equality \eqref{Bourgeoiscond}) and that it is a commuting pair
(resp.\ satisfies property L) if and only if $(A,B)$ is a commuting pair (resp.\ satisfies property L).

\vskip 2mm
Assume now that $A$ satisfies Condition \eqref{condonA} as well as all the previous assumptions,
i.e.\ $(A,B)$ is indecomposable and satisfies Equality \eqref{dSPcond} (resp.\ Equality \eqref{Bourgeoiscond}).
By Lemma \ref{ultimlemma1}, we may choose $k \in \N \setminus \{0\}$ such that $\gamma_k$ is one-to-one.
Replacing $A$ with $kA$, we lose no generality assuming that $\gamma_1$ is one-to-one.

We can conclude: if $e^B$ has several eigenvalues, Lemma \ref{ultimlemma2} contradicts
the assumption that $(A,B)$ is indecomposable.
It follows that $e^B$ has a sole eigenvalue, and for the same reason this is also true of
$e^A$. Choosing $(\alpha,\beta)\in \C^2$ such that $\Sp(e^A)=\{e^\alpha\}$ and
$\Sp(e^B)=\{e^\beta\}$, we find that $\exp(A-\alpha\,I_n)$ and $\exp(B-\beta\,I_n)$
both have $1$ as sole eigenvalue. We deduce that $\Sp(A-\alpha\,I_n) \subset 2i\pi\Z$ and
$\Sp(B-\beta\,I_n) \subset 2i\pi\Z$.
Set $A':=A-\alpha\,I_n$ and $B':=B-\beta\,I_n$.
We now conclude the proofs of Theorem \ref{2matrices} and Proposition \ref{Bourgeoistheo}
by considering the two cases separately. \\
$\bullet$ \textbf{Case 1.} $(A,B)$ satisfies Equality \eqref{dSPcond}.
The pair $(A',B')$ clearly satisfies Equality \eqref{dSPcond}. Proposition \ref{2eetape}
yields that $A'$ commutes with $B'$; hence $AB=BA$. \\
$\bullet$ \textbf{Case 2.} $(A,B)$ satisfies Equality \eqref{Bourgeoiscond}.
The pair $(A',B')$ obviously satisfies Equality \eqref{Bourgeoiscond}. The matrices $e^{A'}$ and $e^{B'}$ commute and are therefore simultaneously triangularizable (see \cite[Theorem 1.1.5]{RadjRosen}).
Moreover, they have $1$ as sole eigenvalue.
Therefore $e^{kA'+B'}=(e^{A'})^ke^{B'}$ has $1$ as sole eigenvalue for every $k \in \N$.
Proposition \ref{BourgeoisimpliqueLrestr} shows that $\bigl(\frac{1}{2i\pi}A',\frac{1}{2i\pi}B'\bigr)$ has property L, which clearly entails that $(A,B)$ has property L.

\vskip 2mm
Thus Theorem \ref{2matrices} and Proposition \ref{Bourgeoistheo} are proven.
\end{proof}

\section{Additive semigroups on which the exponential is a homomorphism}

\begin{Not}
We denote by $\Q_+$ the set of non-negative rational numbers.
\end{Not}

\begin{Def}
A linear subspace $V$ of $\Mat_n(\C)$ has property L when there are $n$ linear forms
$f_1,\dots,f_n$ on $V$ such that
$$\forall M \in V, \; \OSp(M)=\bigl[f_k(M)\bigr]_{1 \leq k \leq n}.$$
\end{Def}

\noindent In this short section, we prove the following result.

\begin{theo}\label{semigroup}
Let $S$ be a subsemigroup of $(\Mat_n(\C),+)$ and assume that $M \mapsto \exp(M)$
is a homomorphism from $(S,+)$ to $(\GL_n(\C),\times)$.
Then $\Vect(S)$ has property L.
\end{theo}

\noindent By Proposition \ref{Bourgeoistheo}, it suffices to establish the following lemma.

\begin{lemme}
Let $S$ be a subsemigroup of $(\Mat_n(\C),+)$. Assume that every pair $(A,B)\in S^2$ has property L.
Then the linear subspace $\Vect(S)$ has property L.
\end{lemme}

\begin{proof}
Let $(A_1,\dots,A_r)$ be a basis of $\Vect(S)$ formed of elements of $S$.
For every $j \in \{1,\dots,r\}$, we choose a list $(a_1^{(j)},\dots,a_n^{(j)})\in \C^n$ such that
$$\OSp(A_j)=\bigl[a_k^{(j)}\bigr]_{1 \leq k \leq n}.$$
Since, for every $(p_1,\dots,p_r)\in \N^r$, the pair $\Bigl(\underset{k=1}{\overset{j-1}{\sum}}p_k A_k,A_j\Bigr)$
has property L for every $j \in \{2,\dots,r\}$, by induction we obtain a list
$(\sigma_1,\dots,\sigma_r)\in (\Sigma_n)^r$ such that
$$\OSp\biggl(\sum_{j=1}^r p_j A_j\biggr)=\biggl[\sum_{j=1}^r p_j\, a^{(j)}_{\sigma_j(k)}\biggr]_{1 \leq k \leq n}.$$
Multiplying by inverses of positive integers, we readily generalize this as follows:
for every $(z_1,\dots,z_r)\in (\Q_+)^r$, there
exists a list
$(\sigma_1,\dots,\sigma_r)\in (\Sigma_n)^r$ such that
$$\OSp\biggl(\sum_{j=1}^r z_j A_j\biggr)=\biggl[\sum_{j=1}^r z_j\, a^{(j)}_{\sigma_j(k)}\biggr]_{1 \leq k \leq n}.$$
Now, we prove the following property, depending on $l \in \{0,\dots,r\}$, by downward induction:
\begin{center}
$\calP(l)$ : For every $(z_1,\dots,z_l) \in (\Q_+)^l$,
there exists a list $(\sigma_1,\dots,\sigma_r) \in (\Sigma_n)^r$ satisfying
$$\forall (z_{l+1},\dots,z_r)\in \C^{r-l}, \;
\OSp\biggl(\sum_{j=1}^r z_j A_j\biggr)=\biggl[\sum_{j=1}^r z_j\, a^{(j)}_{\sigma_j(k)}\biggr]_{1 \leq k \leq n}.$$
\end{center}
In particular, $\calP(r)$ is precisely what we have just proven, whilst
$\calP(0)$ means that there exists a list
$(\sigma_1,\dots,\sigma_r)\in (\Sigma_n)^r$ such that, for every $(z_1,\dots,z_r)\in \C^r$,
$$\OSp\biggl(\sum_{j=1}^r z_j A_j\biggr)=\biggl[\sum_{j=1}^r z_j\, a^{(j)}_{\sigma_j(k)}\biggr]_{1 \leq k \leq n}.$$
Therefore $\calP(0)$ implies that $\Vect(S)$ has property $L$.

Let $l \in \{1,\dots,r\}$ be such that $\calP(l)$ holds, and
fix $(z_1,\dots,z_{l-1}) \in (\Q_+)^{l-1}$. By $\calP(l)$, for every $z_l \in \Q_+$, we may choose a list
$(\sigma_1^{z_l},\dots,\sigma_r^{z_l}) \in (\Sigma_n)^r$ such that
$$\forall (z_{l+1},\dots,z_r)\in \C^{r-l}, \;
\OSp\biggl(\sum_{j=1}^r z_j A_j\biggr)=\biggl[\sum_{j=1}^r z_j\, a^{(j)}_{\sigma_j^{z_l}(k)}\biggr]_{1 \leq k \leq n}.$$
Since $(\Sigma_n)^r$ is finite and $\Q_+ \cap (0,1)$ is infinite, some list
$(\sigma_1,\dots,\sigma_r) \in (\Sigma_n)^r$ equals $(\sigma_1^{z_l},\dots,\sigma_r^{z_l})$
for infinitely many values of $z_l$ in $\Q_+ \cap (0,1)$.
Fixing $(z_{l+1},\dots,z_r) \in \C^{r-l}$, we deduce the identity
$$\forall z_l \in \C, \quad
\chi_{\underset{j=1}{\overset{r}{\sum}} z_j A_j}(X)=\prod_{k=1}^n\biggl(X-\sum_{j=1}^r z_j\, a^{(j)}_{\sigma_j(k)}\biggr)$$
by remarking that, on both sides, the coefficients of the polynomials are polynomials in $z_l$.
Hence
$$\forall (z_l,\dots,z_r)\in \C^{r-l+1}, \;
\OSp\biggl(\sum_{j=1}^r z_j A_j\biggr)=\biggl[\sum_{j=1}^r z_j\, a^{(j)}_{\sigma_j(k)}\biggr]_{1 \leq k \leq n}.$$
This proves that $\calP(l-1)$ holds.
\end{proof}

\section*{Acknowledgements}

The author would like to thank the referee for helping enhance the quality of this article
in a very significant way.

\end{document}